\documentclass{amsart}

\usepackage[utf8]{inputenc}
\usepackage[OT2,T1]{fontenc}
\usepackage[english]{babel}

\usepackage{amssymb,amsthm,amsmath,amsxtra}
\usepackage{comment}
\usepackage{graphicx}
\usepackage{hyperref}  
\usepackage{color}
\usepackage{array}
\usepackage{diagbox}
\usepackage{colonequals}

\usepackage{placeins} 

\usepackage[dvipsnames]{xcolor}

\usepackage{chngcntr}
\usepackage{booktabs}
\usepackage{mathrsfs}

\numberwithin{equation}{subsection}

\theoremstyle{plain}
\newtheorem{thm}[equation]{Theorem}
\newtheorem{prop}[equation]{Proposition}
\newtheorem{lem}[equation]{Lemma}
\newtheorem{cor}[equation]{Corollary}

\newtheorem*{cor*}{Corollary}
\newtheorem*{prob*}{Problem}
\newtheorem*{thm*}{Theorem}
\newtheorem*{thma*}{Theorem A}
\newtheorem*{thmb*}{Theorem B}
\newtheorem*{thmc*}{Theorem C}

\theoremstyle{definition}
\newtheorem{defn}[equation]{Definition}
\newtheorem{alg}[equation]{Algorithm}
\newtheorem{exm}[equation]{Example}

\theoremstyle{remark}
\newtheorem{rmk}[equation]{Remark}

\newenvironment{enumalg}
{\begin{enumerate}}
{\end{enumerate}}
\usepackage[inline, shortlabels]{enumitem}
\newcommand{\defi}[1]{\textit{\textsf{#1}}} 

\DeclareMathOperator{\Aut}{Aut}

\DeclareMathOperator{\Gal}{Gal}

\DeclareMathOperator{\Mon}{Mon}

\DeclareMathOperator{\ord}{ord}

\DeclareMathOperator{\PSL}{PSL}

\DeclareMathOperator{\GL}{GL}

\let\div\relax
\DeclareMathOperator{\div}{div}

\newcommand{\C}{\mathbb C}
\newcommand{\F}{\mathbb F}

\newcommand{\PP}{\mathbb P}
\newcommand{\Q}{\mathbb Q}

\newcommand{\Z}{\mathbb Z}
\newcommand{\Qbar}{{\mathbb Q}^{\textup{al}}}

\renewcommand{\L}{\mathscr L}

\newcommand{\Belyi}{Bely\u{\i}}

\newcommand{\calP}{\mathcal{P}}

\newcommand{\jv}[1]{{\color{red} \sf JV: [#1]}}

\newcommand{\sss}[1]{{\color{green} \sf S$^3$: [#1]}}

\newcommand{\jrs}[1]{{\color{cyan} \sf JRS: [#1]}}

\everymath{\displaystyle}

\begin{document}

\title{A Database of Bely\u{i} Maps}

\author{Michael Musty, Sam Schiavone}
\address{Department of Mathematics, Dartmouth College, 6188 Kemeny Hall, Hanover, NH 03755, USA}
\email{michaelmusty@gmail.com, sam.schiavone@gmail.com}

\author{Jeroen Sijsling}
\address{Universit\"at Ulm, Institut f\"ur Reine Mathematik, D-89068 Ulm, Germany}
\email{jeroen.sijsling@uni-ulm.de}

\author{John Voight}
\address{Department of Mathematics, Dartmouth College, 6188 Kemeny Hall, Hanover, NH 03755, USA}
\email{jvoight@gmail.com}

\date{\today}

\begin{abstract}
  We use a numerical method to compute a database of three-point branched
  covers of the complex projective line of small degree.  We report on some
  interesting features of this data set, including issues of descent.
\end{abstract}

\maketitle

\section{Introduction}

\subsection{Motivation}

Let $X$ be a smooth, projective curve over $\C$.  A \defi{\Belyi\ map} on $X$ is a
nonconstant map $\phi \colon X \to \PP^1$ that is unramified away from
$\{0,1,\infty\}$.  
By a theorem of \Belyi\ \cite{Belyi} and Weil's descent theory \cite{Weil}, $X$ can be defined over the algebraic closure $\Qbar$ of $\Q$ if and only if $X$ admits a \Belyi\ map.  This remarkable
observation has led to a spurt of activity, with many deep questions still open
after forty years.  In his study of covers of the projective line minus three
points \cite{Deligne}, Deligne writes pessimistically:
\begin{quote}
  A.~Grothendieck and his students developed a combinatorial description
  (``maps'') of finite coverings...  
  It has not
  aided in understanding the Galois action.  We have only a few examples of
  non-solvable coverings whose Galois conjugates have been computed.
\end{quote}
Indeed, although significant mathematical effort has been expended in computing
\Belyi\ maps \cite{SijslingVoight}, there have been few systematic computations
undertaken.  

\subsection{Main result} \label{sec:mainresult}

In this article we seek to remedy this state of affairs. We address Deligne's second objection
by describing the uniform computation of a large catalogue of \Belyi\ maps of
small degree.  We utilize the numerical method of
Klug--Musty--Schiavone--Voight \cite{KMSV} and follow the combinatorial
description of Grothendieck. We make some preliminary observations about our data, but leave to future work a more detailed analysis of the Galois action on the maps in our catalogue.

A \defi{passport} is the data $(g,G,\lambda)$ consisting of a nonnegative
integer $g \in \Z_{\geq 0}$, a transitive permutation group $G \leq S_d$, and
three partitions $\lambda=(\lambda_0,\lambda_1,\lambda_\infty)$ of $d$. The
\defi{passport} of a \Belyi\ map is given by its genus, its monodromy group,
and the ramification degrees of the points above $0,1,\infty$. There is a
natural permutation action of $S_3$ on passports, so (without loss of
generality) we choose exactly one passport up to this $S_3$-action.  (For more
on passports, see section \ref{sec:passports}.)

A summary of the scope of our computation so far is given in \eqref{table:results}: we
list the number of passports of \Belyi\ maps for each degree $d$ and genus $g$
as well as the number of them that we have computed ({\color{ForestGreen}green} number).
Our data is available at \url{https://github.com/michaelmusty/BelyiDB}
and will hopefully also be available at \url{lmfdb.org} in the near future.

\subsection{Comparison}

Our database compares to the existing catalogues of \Belyi\ maps that are currently available as follows.
\begin{itemize}
  \item Birch \cite{Birch} computed a sampling of \Belyi\ maps of low degree
    and genus, for a total of $50$ passports.
  \item A \defi{Shabat polynomial} is a \Belyi\ map of genus $0$ that is
    totally ramified at $\infty$.  B\'etr\'ema--P\'er\'e--Zvonkin
    \cite{Betrema} computed all \emph{Shabat polynomials} up to degree $8$:
    there are
    $78$ such passports.
  \item A \Belyi\ map is \defi{clean} if every point above $1$ has ramification
    index $2$.   (A clean \Belyi\ map has even degree, and if $\phi$ is an
    arbitrary \Belyi\ map of degree $d$ then $4\phi(1-\phi)$ is a clean \Belyi\
    map of degree $2d$.) Adrianov~et~al.~\cite{Adrianov} computed
    all clean \Belyi\ maps up to degree $8$: there are 
    $67$ such passports.
  \item Malle \cite{Malle} computed fields of definition of some genus zero 
    passports whose permutation group is primitive, subject to some other 
    restrictions, up to degree $13$: there are hundreds of passports.
  \item
    Bose--Gundry--He \cite{bose} describe a partial catalogue of \Belyi\
    maps, inspired by considerations from gauge theory in physics. This database contains many genus $0$ maps
    up to degree $7$ and some maps in genus $1$ and $2$.
\end{itemize}
There are many other papers that compute certain classes of \Belyi\ maps:
for further reference, see
Sijsling--Voight \cite{SijslingVoight}.

\begin{equation} \label{table:results}
  \begin{tabular}{c||c|c|c|c|c|| c}
    \diagbox{$d$}{$g$} & $0$ & $1$ & $2$ & $3$ & $\geq 4$ & total\\
    \hline\hline
    $1$\rule{0pt}{2.5ex} & {\color{ForestGreen}1}$/1$ & $0$ & $0$ & $0$ & $0$ & {\color{ForestGreen}1}$/1$\\ 
    $2$ & {\color{ForestGreen}1}$/1$ & $0$ & $0$ & $0$ & $0$ & {\color{ForestGreen}1}$/1$\\ 
    $3$ & {\color{ForestGreen}2}$/2$ & {\color{ForestGreen}1}$/1$ & $0$ & $0$ & $0$ & {\color{ForestGreen}3}$/3$\\ 
    $4$ & {\color{ForestGreen}6}$/6$ & {\color{ForestGreen}2}$/2$ & $0$ & $0$ & $0$ & {\color{ForestGreen}8}$/8$\\ 
    $5$ & {\color{ForestGreen}12}$/12$ & {\color{ForestGreen}6}$/6$ & {\color{ForestGreen}2}$/2$ & $0$ & $0$ & {\color{ForestGreen}20}$/20$\\ 
    $6$ & {\color{ForestGreen}38}$/38$ & {\color{ForestGreen}29}$/29$ & {\color{ForestGreen}7}$/7$ & $0$ & $0$ & {\color{ForestGreen}74}$/74$\\ 
    $7$ & {\color{ForestGreen}89}$/89$ & {\color{ForestGreen}50}$/50$ & {\color{ForestGreen}7}$/13$ & {\color{ForestGreen}2}$/3$ & $0$ & {\color{ForestGreen}148}$/155$\\ %
    $8$ & {\color{ForestGreen}243}$/261$ & {\color{ForestGreen}83}$/217$ & {\color{ForestGreen}0}$/84$ & {\color{ForestGreen}0}$/11$ & $0$ & {\color{ForestGreen}326}$/573$\\ 
    $9$ & {\color{ForestGreen}410}$/583$ & {\color{ForestGreen}33}$/427$ & {\color{ForestGreen}0}$/163$ & {\color{ForestGreen}0}$/28$ & {\color{ForestGreen}0}$/6$ & {\color{ForestGreen}443}$/1207$\\ 
    \hline
    total & {\color{ForestGreen}802}$/993$ & {\color{ForestGreen}204}$/732$ & {\color{ForestGreen}16}$/269$ & {\color{ForestGreen}2}$/42$ & {\color{ForestGreen}0}$/6$ & {\color{ForestGreen}1024}$/2042$
  \end{tabular}
\end{equation}

\subsection{Outline}

The paper is organized as follows.  We begin in section \ref{sec:passports} by
defining passports and exhibiting an algorithm to enumerate their
representative permutation triples up to simultaneous conjugation.  In section
\ref{sec:KMSV}, we briefly recall the numerical method employed.  In section
\ref{sec:descent}, we treat the descent issues that arise.  In sections
\ref{sec:genusone}--\ref{sec:highergenus}, we detail steps that are specific to
elliptic and hyperelliptic curves, and provide examples of these computations.
We conclude in section \ref{sec:database} with a description of the database,
some statistics, and some final observations.

\subsection{Acknowledgements}
The authors would like to thank Hartmut Monien and Greg Warrington for useful
conversations; Mauricio Esquivel Rogel for his implementation of some numerical linear algebra routines, supported by a James O.~Freedman Presidential Scholarship; and Joshua Perlmutter for help in verification.  Thanks also to Maarten Derickx, Noam Elkies, and David P.\ Roberts for comments.  Our calculations are performed in the computer algebra system \textsf{Magma} \cite{Magma}.  Voight was supported by an NSF CAREER Award (DMS-1151047) and a Simons Collaboration Grant (550029).

\section{Passports} \label{sec:passports}

We begin by explaining the combinatorial (or topological) description of \Belyi\
maps and exhibit an efficient method for their enumeration.  For general
background reading, see Sijsling--Voight \cite[\S 1]{SijslingVoight} and the
references therein.

\subsection{Preliminaries} \label{subsec:intro}

Throughout, let $K \subseteq \C$ be a field.  A \defi{(nice) curve} over $K$ is
a smooth, projective, geometrically connected (irreducible) scheme of finite
type over $K$ that is pure of dimension $1$. After extension to $\C$, a curve
may be thought of as a compact, connected Riemann surface.  A \defi{\Belyi\
map} over $K$ is a finite morphism $\phi\colon X \to \PP^1$ over $K$ that is
unramified outside $\{0,1,\infty\}$; we will sometimes write $(X,\phi)$ when we
want to pay special attention to the source curve $X$.  Two \Belyi\ maps
$\phi,\phi'$ are \defi{isomorphic} if there is an isomorphism $\iota\colon X
\xrightarrow{\sim} X'$ of curves such that $\phi'\iota=\phi$.

Let $\phi\colon X\to\PP^1$ be a \Belyi\ map over $\Qbar$ of degree $d \in
\Z_{\geq 1}$.  The \defi{monodromy group} of $\phi$ is the Galois group
$\Mon(\phi) \colonequals \Gal(\C(X)\,|\,\C(\PP^1)) \leq S_d$ of the
corresponding extension of function fields (understood as the action of the
automorphism group of the normal closure); the group $\Mon(\phi)$ may also be
obtained by lifting paths around $0,1,\infty$ to $X$.

A \defi{permutation triple} of degree $d \in \Z_{\geq 1}$ is a tuple $\sigma =
(\sigma_0,\sigma_1,\sigma_\infty)\in S_d^3$ such that $\sigma_\infty \sigma_1
\sigma_0 = 1$. A permutation triple is \defi{transitive} if the subgroup
$\langle \sigma \rangle \leq S_d$ generated by $\sigma$ is transitive.  We say
that two permutation triples $\sigma,\sigma'$ are \defi{simultaneously
conjugate} if there exists $\tau\in S_d$ such that
\begin{equation}\label{eqn:simconj}
  \sigma^\tau \colonequals
  (\tau^{-1}\sigma_0\tau, \tau^{-1}\sigma_1\tau, \tau^{-1}\sigma_\infty\tau)
  = \left(\sigma'_0,\sigma'_1,\sigma'_\infty\right)
  = \sigma'.
\end{equation}
An automorphism of a permutation triple $\sigma$ is an element of $S_d$ that
simultaneously conjugates $\sigma$ to itself, i.e.,
$\Aut(\sigma)=Z_{S_d}(\langle \sigma \rangle)$, the centralizer inside $S_d$.

\begin{lem} \label{lem:simulisom}
  The set of transitive permutation triples of degree $d$ up to simultaneous
  conjugation is in bijection with the set of \Belyi\ maps of degree $d$ up to
  isomorphism.
\end{lem}

\begin{proof}
  The correspondence is via monodromy \cite[Lemma 1.1]{KMSV}; in particular,
  the monodromy group of a \Belyi\ map is (conjugate in $S_d$ to) the group
  generated by~$\sigma$.
\end{proof}

The group $\Gal(\Qbar\,|\,\Q)$ acts on \Belyi\ maps by acting on the
coefficients of a set of defining equations; under the bijection of Lemma
\ref{lem:simulisom}, it thereby acts on the set of transitive permutation
triples, but this action is rather mysterious.

We can cut this action down to size by identifying some basic invariants, as
follows.  A \defi{passport} consists of the data $\calP=(g,G,\lambda)$ where $g
\geq 0$ is an integer, $G \leq S_d$ is a transitive subgroup, and
$\lambda=(\lambda_0,\lambda_1,\lambda_\infty)$ is a tuple of partitions
$\lambda_s$ of $d$ for $s=0,1,\infty$.  These partitions will be also be
thought of as a tuple of conjugacy classes $C=(C_0,C_1,C_\infty)$ by cycle
type, so we will also write passports as $(g,G,C)$. The \defi{passport} of a
\Belyi\ map $\phi\colon X \to \PP^1$ is $(g(X),\Mon(\phi),
(\lambda_0,\lambda_1,\lambda_\infty))$, where $g(X)$ is the genus of $X$ and
$\lambda_s$ is the partition of $d$ obtained by the ramification degrees above
$s=0,1,\infty$, respectively. Accordingly, the \defi{passport} of a transitive
permutation triple $\sigma$ is $(g(\sigma),\langle \sigma \rangle,
\lambda(\sigma))$, where (by Riemann--Hurwitz)
\begin{equation}
  g(\sigma) \colonequals 1-d+(e(\sigma_0)+e(\sigma_1)+e(\sigma_\infty))/2
\end{equation}
and $e$ is the index of a permutation ($d$ minus the number of orbits), and
$\lambda(\sigma)$ is the cycle type of $\sigma_s$ for $s=0,1,\infty$. The
\defi{size} of a passport $\calP$ is the number of simultaneous conjugacy
classes (as in \ref{eqn:simconj}) of (necessarily transitive) permutation
triples $\sigma$ with passport $\calP$.

The action of $\Gal(\Qbar\,|\,\Q)$ on \Belyi\ maps preserves passports.
Therefore, after computing equations for all \Belyi\ maps with a given
passport, we can try to identify the Galois orbits of this action.  We say a
passport is \defi{irreducible} if it has one $\Gal(\Qbar\,|\,\Q)$-orbit and
\defi{reducible} otherwise.

\subsection{Passport lemma}

To enumerate passports, we will use the following lemma.

\begin{lem}\label{lem:passportlemma}
  Let $S$ be a group, let $G \leq S$ be a subgroup, let $N \colonequals N_S(G)$
  be the normalizer of $G$ in $S$, and let $C_0,C_1$ be conjugacy classes in
  $N$ represented by $\tau_0,\tau_1 \in G$.  Let $Z_N(g)$ denote the
  centralizer of $g$ in $N$.
  Let
  \begin{equation}
    U \colonequals \{(\sigma_0,\sigma_1) \in C_0 \times C_1:
    \langle \sigma_0, \sigma_1 \rangle \subseteq G \}/\!\sim
  \end{equation}
  where $\sim$ indicates simultaneous conjugation by elements in $S$.
  Then the map
  \begin{equation} \label{eqn:ZNmap}
    \begin{aligned}
      u\colon Z_N(\tau_0) \backslash N / Z_N(\tau_1) &\to U \\
      Z_N(\tau_0)\nu Z_N(\tau_1) &\mapsto [(\tau_0,\nu\tau_1\nu^{-1})]
    \end{aligned}
  \end{equation}
  is surjective, and for all $[(\sigma_0,\sigma_1)] \in U$ such that $\langle
  \sigma_0,\sigma_1 \rangle = G$, there is a unique preimage under $u$.
\end{lem}

\begin{proof}
  The map \eqref{eqn:ZNmap} is well-defined, as $\nu \in N$ so
  $\nu\tau_1\nu^{-1} \in G$ and conjugacy classes are taken in $N$.

  We first show that \eqref{eqn:ZNmap} is surjective. Let $[(\sigma_0,
  \sigma_1)] \in U$. Then $g \sigma_0 g^{-1} = \tau_0$ for some $g \in N$, and
  so $[(\sigma_0,\sigma_1)]=[(\tau_0, g \sigma_1 g^{-1})] \in U$. Similarly,
  there is $h \in N$ such that $\sigma_1 = h\tau_1 h^{-1}$ so
  $[(\sigma_0,\sigma_1)]=[(\tau_0,(gh) \tau_1 (gh)^{-1})]$, and $gh=\nu \in N$.

  Next, we show \eqref{eqn:ZNmap} is injective when restricted to
  generating pairs. Suppose $[(\tau_0, \nu
  \tau_1 \nu^{-1})] = [(\tau_0, \mu \tau_1 \mu^{-1})] \in U$ with $\mu,\nu \in
  N$. Then there exists $\rho \in S$ with
  \begin{equation}
    \rho(\tau_0,\nu\tau_1 \nu^{-1})\rho^{-1} = (\tau_0,\mu\tau_1 \mu^{-1}).
  \end{equation}
  Since then $\rho \langle \tau_0, \nu\tau_1 \nu^{-1} \rangle \rho^{-1} = \rho
  G \rho^{-1} = \langle \tau_0, \mu \tau_1 \mu^{-1} \rangle = G$ under the
  hypotheses on generation, so we have $\rho \in N$. The equation in the first
  component reads $\rho\tau_0 \rho^{-1} = \tau_0$, so $\rho \in Z_N(\tau_0)$ by
  definition. The second equation yields
  \begin{equation}
    \begin{aligned}
      \rho \nu \tau_1 \nu^{-1} \rho^{-1} &= \mu\tau_1\mu^{-1} \\
      (\mu^{-1}\rho\nu)\tau_1(\mu^{-1}\rho\nu)^{-1} &= \tau_1
    \end{aligned}
  \end{equation}
  so $\mu^{-1} \rho \nu \in Z_N(\tau_1)$.  Writing $\nu = (\rho^{-1}) \mu
  (\mu^{-1} \rho \nu)$, we find that $Z_N(\tau_0) \nu Z_N(\tau_1) = Z_N(\tau_0)
  \mu Z_N(\tau_1)$, as desired.
\end{proof}

\subsection{Computing passports}

We now describe an algorithm to produce all passports for a given degree $d$
and a representative set of permutation triples in each passport up to
simultaneous conjugation.  We simplify this description by considering the
transitive subgroups of $S_d$ one at a time: these are currently available
\cite{Magma} up to degree $47$.

There is a natural 
permutation action of $S_3$ on passports and on the permutation triples in a passport,
corresponding to postcomposition of \Belyi\ maps by an automorphism of the base curve $\PP^1$
permuting $\{0,1,\infty\}$. 
For the purposes of tabulation, we will choose one
passport up to this action of $S_3$: to do so, we choose a total ordering
$\preceq$ on partitions (e.g., refining the dominance partial order).

\begin{alg} \label{alg:passportalg}
  Let $d \in \Z_{\geq 1}$, let $G\leq S_d$ be a transitive subgroup, and let $N
  \colonequals  N_{S_d}(G)$ be the normalizer of $G$ in $S_d$. This algorithm
  returns a representative list of passports for $G$ up to the action of $S_3$;
  and, for each passport, a representative list of permutation triples (one for
  each simultaneous conjugacy class).
  \begin{enumalg}
    \item Compute representatives $\{\tau_1,\dots,\tau_r\}$ for the conjugacy
      classes $\{C_1,\dots, C_r\}$ of $G$ up to conjugation by $N$.
    \item Out of the $r^2$ possible pairs of conjugacy class representatives,
      only consider pairs $(\tau_i,\tau_j)$ with
      $\lambda(\tau_i)\preceq\lambda(\tau_j)$.
    \item For each pair $(\tau_i,\tau_j)$ from Step 2, apply Lemma
      \ref{lem:passportlemma} to compute the set
      \begin{equation}
        U_{ij} \colonequals \{(\sigma_0,\sigma_1) \in C_i \times C_j:
        \langle\sigma_0,\sigma_1\rangle \subseteq G \}/\!\sim
      \end{equation}
      by computing the double coset $Z_N(\tau_i)\backslash N / Z_N(\tau_j)$ and
      applying the map $u$.  Complete each pair $(\sigma_0,\sigma_1)\in U_{ij}$
      to a permutation triple by setting $\sigma_\infty \colonequals
      (\sigma_1\sigma_0)^{-1}$, and let $T_{ij}$ denote the resulting set of
      triples obtained from $U_{ij}$.
    \item Keep only those triples $\sigma \in T_{ij}$ with $\langle \sigma
      \rangle = G$ and such that $\lambda(\sigma_1) \preceq \lambda
      (\sigma_\infty)$.
    \item Sort the triples obtained from Step 4 into passports by cycle
      structure.
  \end{enumalg}
\end{alg}

\begin{proof}[Proof of correctness]
  We compute all possible input pairs $(\tau_0,\tau_1)$ to Lemma
  \ref{lem:passportlemma} with $\lambda(\tau_0)\preceq\lambda(\tau_1)$. This
  accounts for all possible input pairs to Lemma \ref{lem:passportlemma} since
  every passport is $S_3$-equivalent to such a passport. We do not have control
  over the conjugacy class of $\sigma_\infty$ in this process, but Step 4
  insists that every resulting passport representative $\sigma$ has
  $\lambda(\sigma_0)\preceq\lambda(\sigma_1)\preceq\lambda(\sigma_\infty)$
  thereby ensuring a unique passport up to the action of $S_3$.
\end{proof}

We computed representatives for all passports (without equations) in degree $d\leq 11$
using Algorithm \ref{alg:passportalg}: this took about $18$ minutes for all degrees $d \leq 9$,
about $3.3$ hours for $d=10$, and $2.37$ days for $d=11$.

\section{Numerical computation of Belyi maps} \label{sec:KMSV}

With triples and passports in hand, we now briefly review the numerical method
used to compute \Belyi\ maps.

\subsection{Overview} \label{sec:overview}

The method of Klug--Musty--Schiavone--Voight \cite{KMSV} takes as input a
permutation triple $\sigma = (\sigma_0, \sigma_1, \sigma_{\infty})$ and
produces as output equations for the curve $X$ and \Belyi\ map $\phi\colon X
\to \PP^1$ over a number field $K \subseteq \C$ that corresponds to $\sigma$
(in the monodromy bijection of Lemma \ref{lem:simulisom}).

This method is numerical, so it is not guaranteed to terminate (because of loss
of precision or convergence issues), but when it terminates, it gives correct
output.  The method proceeds in the following steps.

\begin{enumalg}
  \item Form the triangle subgroup $\Gamma \leq \Delta(a,b,c)$ associated to
    $\sigma$ and compute its coset graph.

  \item Use a reduction algorithm for $\Gamma$ and numerical linear algebra to
    compute power series expansions of modular forms $f_i \in S_k(\Gamma)$ for
    an appropriate weight $k$.

  \item Use numerical linear algebra (and Riemann--Roch) to find polynomial
    relations among the series $f_i$ to compute equations for the curve $X$ and
    similarly to express the map $\phi$ in this model.

  \item Normalize the equations of $X$ and $\phi$ so that the coefficients are
    algebraic; recognize these coefficients as elements of a number field $K
    \subseteq \C$.

  \item Verify that $\phi$ has the correct ramification and monodromy.
\end{enumalg}

For the purposes of this article, the reader may treat this method as a black
box with two exceptions: in section \ref{sec:ahhhTheta} we describe an
improvement to the method in Step 4 for a choice of descent constant, and we
discuss a numerical test for hyperellipticity using power series in weight $2$
in section \ref{sec:numhyperl}.

\subsection{Discussion}

There are a few key advantages of the above algorithm for our purposes.  First,
it is uniform, and in particular does not require the permutation triple to
have a special form or for the curve to be of any particular genus.  Second, it
computes one \Belyi\ map at a time, without needing the whole passport: and in
particular, there are no \emph{parasitic solutions} (degenerate maps that arise
in other computational methods).  Third, we obtain the bijection between
triples and \Belyi\ maps by the very construction of the equations (and the
embedding $K \hookrightarrow \C$).

There is an alternative method due to Monien \cite{Monien, Co3} that uses
noncocompact triangle subgroups $\Gamma \leq \Delta(2,3,\infty) \simeq
\PSL_2(\Z)$ instead of our cocompact subgroups.  This method has been shown to
work in genus zero for maps of very large degree
(e.g. a \Belyi\ map with monodromy group isomorphic to
the Conway group $\mathrm{Co}_3$ is given in \cite{Co3}).

\section{Descent issues}\label{sec:descent}

In this section, we discuss issues of descent for \Belyi\ maps: when can a
\Belyi\ map be defined over a minimal field?  (The reader eager for \Belyi\ map
computations should skip this and proceed to the next section.)  A
satisfactory answer to this question is crucial for understanding the action
of $\Gal(\Qbar\,|\,\Q)$ on \Belyi\ maps.

\subsection{Field of moduli and field of definition}

Let $\sigma$ be a permutation triple with passport $\calP$ and corresponding
\Belyi\ map $\phi\colon X \to \PP^1$ over $\Qbar$.  The \defi{field of moduli}
$M(X,\phi) \subseteq \Qbar \subset \C$ of $\phi$ is the fixed field of $\{
\tau \in \Gal(\Qbar \,|\, \Q) : \tau (\phi) \simeq \phi \}$.  The field of
moduli is the intersection of all fields over which $(X,\phi)$ can be defined.

The degree of $M(X,\phi)$ is bounded above by the size of the passport $\calP$;
this bound is achieved if and only if the passport is irreducible.

\begin{defn}
  We say that $(X,\phi)$ \defi{descends} (to its field of moduli) if $(X,\phi)$ can be
  defined over its field of moduli $M(X,\phi)$, that is, if there exists a
  \Belyi\ map $\phi_K\colon X_K \to \PP^1$ over $K$ whose base change to
  $\Qbar$ is isomorphic to $\phi\colon X \to \PP^1$.
\end{defn}

Weil \cite{Weil} studied general conditions for descent.  For example, if
$\phi$ has trivial automorphism group $\Aut(\phi)$, then $\phi$ descends---this
criterion suffices to deal with a large majority of \Belyi\ maps.  More
generally, to descend the \Belyi\ map it is necessary and sufficient to
construct a \defi{Weil cocycle}, a collection of isomorphisms $f_{\sigma}\colon
\sigma(X) \to X$, one for every element $\sigma \in \Gal_{M(X,\phi)}
\colonequals \Gal(\Qbar \,|\, M(X,\phi))$, such that $f_{\sigma\tau} = f_\sigma
\sigma(f_{\tau})$ for all $\sigma, \tau \in \Gal_{M(X,\phi)}$.  (When
$\Aut(\phi)$ is trivial, this condition is satisfied for any collection of
isomorphisms $f_\sigma$.)  This criterion can be made explicit and computable
\cite[Method 4.1]{SijslingVoight}.


\subsection{Pointed descent}

There is another way to sidestep descent issues by rigidifying, as follows.

\begin{defn}
  A \defi{pointed \Belyi\ map} $(X, \phi; P$) is a \Belyi\ map $(X, \phi)$
  together with a point $P \in \phi^{-1}(\{0,1,\infty\}) \subseteq X(\Qbar)$.
  An isomorphism of pointed \Belyi\ maps $(X, \phi; P) \xrightarrow{\sim} (X',
  \phi'; P')$ is an isomorphism of \Belyi\ maps $\iota$ such that
  $\iota(P)=P'$.
\end{defn}


\begin{rmk}
  In our computations we choose the point $P$ to be one of the ramification points of $\phi$. Any point on $X$ would do, but only the ramification points can be seen from
  the combinatorial data.
\end{rmk}

\begin{defn}
  A \defi{pointed permutation triple} $(\sigma ; c)$ is a permutation triple
  $\sigma \in S_d^3$ together with a distinguished cycle $c$ in one of the
  permutations $\sigma_s$ with $s=0,1,\infty$; we call $s$ its \defi{base
  point} and the length of the cycle $c$ its \defi{length}. We call
  $(\sigma;c)$ a \defi{pointed refinement} of the permutation triple $\sigma$.

  Two pointed permutation triples $(\sigma; c)$ and $(\sigma';c')$ are
  \defi{simultaneously conjugate} if the permutation triples $\sigma,\sigma'$
  are simultaneously conjugate by an element $\tau \in S_d$ such that $c^\tau =
  c'$.  The automorphism group $\Aut(\sigma;c) \leq \Aut(\sigma)$ is the
  subgroup of $S_d$ that simultaneously conjugates $(\sigma;c)$ to itself.
\end{defn}

Returning to the correspondence of Lemma \ref{lem:simulisom}, we see that
pointed permutation triples of degree $d$ up to simultaneous conjugation are in
bijection with pointed \Belyi\ maps of degree $d$ up to isomorphism.

\begin{prop}
  The base point, length, and cardinality of the automorphism group of a
  pointed permutation triple are invariant under simultaneous conjugation and
  under the action of $\Gal(\Qbar\,|\,\Q)$.
\end{prop}

\begin{proof}
  The statements for simultaneous conjugation are clear.  For the Galois
  action, we pass back to \Belyi\ maps: the base point, the 
  ramification index, and the automorphism group of a pointed \Belyi\ map
  are Galois invariant.
\end{proof}

We similarly define the field of moduli $M(X,\phi;P)$ for a pointed \Belyi\ map.
The following theorem gives us a widely applicable criterion for
descent (even in the presence of automorphisms).

\begin{thm}\label{th:ptdesc}
  A pointed \Belyi\ map $(X,\phi;P)$ descends, i.e., the curve $X$, the map
  $\phi$, and the point $P$ can all be defined over $M(X,\phi;P)$.
\end{thm}

\begin{proof}
  The statement is given by Birch \cite[Theorem 2]{Birch}; for a constructive
  proof using branches, see Sijsling--Voight \cite[Theorem
  1.12]{SijslingVoightBirch}.
\end{proof}

\subsection{Pointed passports}

Given the simplicity and importance of Theorem \ref{th:ptdesc}, we refine our
notion of passport as follows.

\begin{defn}
  A \defi{pointed passport} is the data $(g,G,\lambda;c)$ where $(g,G,\lambda)$
  is a passport and $c=(s,e,a)$ consists of the data: $s \in \{0,1,\infty\}$,
  and $e \in \Z_{\geq 1}$ a summand in the partition $\lambda_s$, and finally
  $a \in \Z_{\geq 1}$.
\end{defn}

Given a pointed \Belyi\ map $(X,\phi;P)$, we define its \defi{pointed passport}
$\calP(X,\phi;P)$ to be its passport together with the data $s=\phi(P)$, the
ramification degree $e=e_\phi(P)$, and $a=\#\!\Aut(X,\phi;P)$.  Likewise, we define 
the pointed passport $\calP(\sigma;c)$ to be the passport with $s$ its base point, $e$ its length,
and $a$ the cardinality of its automorphism group.  We define the \defi{size}
of a pointed passport $\calP$ to be the number of isomorphism classes of
pointed \Belyi\ maps (equivalently, number of classes of pointed permutation
triples) with pointed passport $\calP$.

\begin{cor}\label{cor:fodbound}
  A pointed \Belyi\ map is defined over a field whose degree is at most the
  size of its pointed passport.
\end{cor}

\begin{proof}
  Apply Theorem \ref{th:ptdesc}.
\end{proof}

\begin{prop}\label{prop:uniquep}
  If the size of $\calP(\sigma;c)$ is equal to the size of $\calP(\sigma)$,
  then all \Belyi\ maps with passport $\calP(\sigma)$ descend.
\end{prop}

\begin{proof}
  Any field of definition of a pointed \Belyi\ map is also a field of
  definition of the underlying \Belyi\ map, so the fields of moduli and pointed
  moduli coincide by hypothesis.  Descent follows by Theorem \ref{th:ptdesc},
  since the moduli field of the pointed curve is a field of definition.
\end{proof}

It seems quite common for a permutation triple to have a pointed refinement
of size $1$.  The first example where no such refinement exists occur in
degree $8$: see Example \ref{ex:deg8} below.


\subsection{Descent from $\C$} \label{sec:ahhhTheta}

In Step 4 of our numerical method (see section \ref{sec:overview}), there is a
normalization procedure which we may interpret as an application of pointed
descent, as follows.  In the original method \cite[\S 5]{KMSV}, modular forms are expanded as power
series centered in a neighborhood of a ramification point of the form $|w| < 1$
in a parameter $w$, and the coefficients of these power series are renormalized
by writing them in terms of $\Theta w$ for a certain transcendental factor
$\Theta$, computed as the ratio of two `consecutive' terms in the power series
expansion.  In our setting, we instead normalize not the coefficients of the power series
but instead coefficients of the \Belyi\ map itself, now setting `consecutive'
coefficients equal.  In practice, we find that this normalization requires
smaller precision to recognize the \Belyi\ map exactly from its numerical
approximation.

\begin{rmk}
  In every example we computed, and in both ways of normalizing, we obtained
  normalized power series expansions that numerically agree with series defined
  over $M(X,\phi;P)$ with chosen ramification point $P$.  Currently this is
  only a numerical observation; but it is a sensible expectation, as the method
  works by computing the pluricanonical image using expansions at the designated point.
\end{rmk}

%

\begin{exm}
  Consider the passport $(1,S_5,(5^1, 4^1 1^1, 4^1 1^1))$.  The unique
  representative up to simultaneous conjugation is given by $\sigma$ with
  \begin{equation}
    \sigma_0 = (1\ 4\ 2\ 5\ 3), \quad
    \sigma_1 = (1\ 2\ 3\ 4), \quad
    \sigma_\infty = (1\ 2\ 3\ 5) \, .
  \end{equation}
  We take $c=(1\ 4\ 2\ 5\ 3)$, which has length $5$ and trivial automorphism
  group.
  Since the pointed passport also has size 1,
  the field of moduli of the \Belyi\ map equals $\Q$ by Corollary \ref{cor:fodbound},
  and we can descend to this field by Proposition \ref{prop:uniquep}.
%
  Computing with $50$ digits of precision (here and throughout, we only ever
  display $5$ digits), we find $X\colon y^2 = x^3 - 27 c_4 x - 54 c_6$ with
  \begin{equation}
    c_4 \approx 0.01030 + 0.00748 i
    \qquad \qquad
    c_6 \approx -0.00270 + 0.00196 i
  \end{equation}
  and \Belyi\ map $\phi$ with:
  \begin{equation}
    \phi \approx \frac{2.0000}{-1+(2.21275+0.71897i) y + (1.77422 i) xy}
    = \frac{2}{-1+b_3y+b_5xy}
  \end{equation}
  (where $i^2=-1$).
  The indeterminacy in this approximation is by $\lambda \in
  \C^\times$, acting according to the degree of the pole at $\infty$, so
  $(c_4,c_6) \leftarrow (\lambda^{-4} c_4, \lambda^{-6} c_6)$ and $(x,y)
  \leftarrow (\lambda^{-2} x, \lambda^{-3} y)$. Taking
  \begin{equation}
    \lambda \colonequals \frac{b_5}{b_3^2} \approx -0.19265-0.26516i
  \end{equation}
  the rescaled values $b_3' \colonequals \lambda^3 b_3 \approx 2^{16}/5^{10}$
  and $b_5' \colonequals \lambda^5 b_5 \approx -2^8/5^5$ have $(b_3')^2/b_5' =
  1$ (and there exists a descent with this ratio, defined over $\Q$). Now all
  the coefficients $a_0,b_3,b_5,c_4,c_6 \in \Q$ are easily identified. After
  computing a minimal model and swapping $0, \infty \in \PP^1$ for
  cosmetics, we obtain $X\colon y^2 = x^3 + 5x + 10$ with map
  \begin{equation}
    \phi(x,y) = ((x - 5)y+16)/32 \, .
  \end{equation}
\end{exm}

\subsection{Examples}

We now discuss some examples to see the various subtleties that play a role
when descending \Belyi\ maps.

\begin{exm}\label{ex:deg8}
  The first case of a passport for which Proposition \ref{prop:uniquep} does
  not apply occurs in degree $8$, given by $(1,V_4^2:S_3,(3^21^2,4^2,4^2))$.
  The passport is size $1$ but all pointed passports are size $2$.
  The \Belyi\ map descends because its automorphism group is trivial. A
  descent is given by $X\colon y^2 = x^3+x^2+8x+8$ and
  \begin{equation*}
    \phi(x, y) =\frac{4(7x^4 + 24x^3 + 92x^2 + 320x + 272) y - 16
    (x+1)(x^2+8)(x^2+16x+24)}{27 x^4 y}.
  \end{equation*}
  Because $\Aut(X,\phi)$ is trivial, this is the only model
  over $\Q$ up to isomorphism.  Finally, none of its ramification points is rational, so
  no descent of a pointed refinement immediately gets us to the field of moduli
  $\Q$.
\end{exm}

\begin{exm}
  The first dessin that does not descend to its field of moduli is of degree
  $16$.  Indeed, in lower degree, there are only three passports for which Proposition
  \ref{prop:uniquep} does not apply \emph{and} the automorphism group is
  nontrivial: all occur in degree $12$, one with size $1$, the other two of
  size $2$. Yet explicit calculation shows that these three examples all
  descend.

  For purposes of illustration, we consider the passport
  $(4,\textup{t12n57},(6^2,6^2,6^2))$ of size $2$, where $\textup{t12n57}$ denotes the transitive group in $S_{12}$ numbered $57$. 
  The passport is irreducible and the curves are nonhyperelliptic: they arise
  as degree $2$ covers branching at the ramification points of the unique
  \Belyi\ map with passport $(1,A_4(6),(3^2,3^2,3^2))$, given by $E \colon y^2=x^3+6x^2-3x$
  and \Belyi\ map $\phi(x,y)=(x^2+3)y/(8x^2)$.  
  The ramification points are then exactly the $\Q$-rational points $\infty$,
  $(0, 0)$, $(1, \pm 2)$, $(-3, \pm 6)$ on $E$. To construct the resulting
  degree $2$ cover, we choose a $4$-torsion point $P_4$ on $E$. Then the sum of
  the ramification points and $2 P_4$ is equivalent to $8 \infty$, so that we
  get a function whose square root gives rise to the requested cover. The four
  possible covers thus obtained are all Galois conjugate; we get the same
  \Belyi\ map, this from the curve
  \begin{equation}
    X \colon
    \begin{array}{l}
      y^2 = x^3 + 6 x^2 - 3 x, \\
      w^2 = y x^2 + 2 y x - 3 y + \alpha x^3 + 2 \alpha x^2 - 3 \alpha x ,
    \end{array}
  \end{equation}
  where $\alpha^4 - 12 \alpha^2 + 48=0$. The field $\Q (\alpha)$ contains $\Q
  (\sqrt{-3})$. This unique quadratic subfield is also the field of moduli of
  the \Belyi\ map from $X$, since one can show that it is mapped to its $\Q
  (\sqrt{-3})$-conjugate by the automorphism
  \begin{equation}\label{eq:twisttotake}
    (x, y, w) \mapsto \left(\frac{-3}{x}, \frac{3 y}{x^2}, \frac{3 iw}{x^2}\right)
  \end{equation}
  of $X$. To show that the \Belyi\ map descends, it suffices \cite[Cor.\
  5.4]{DebesEmsalem} (or \cite[Theorem 3.4.8]{SijslingVoightBirch}
  with $\mathcal{R} = \emptyset$) to show that the canonical model $E_0$ of $E$
  corresponding to the cocycle defined by the first two entries of
  \eqref{eq:twisttotake} has a rational point. It does; in fact $E_0$ is
  isomorphic to $E$. Still, none of the points on $E_0$ that correspond to the
  ramification points of $E$ are rational over $\Q (\sqrt{-3})$.  We conclude
  that there is no choice of \emph{pointed refinement} that will give rise to a
  descent to $\Q(\sqrt{-3})$ in this case, even though the \Belyi\ map
  descends.


%
\end{exm}

\section{Genus one} \label{sec:genusone}

In this section, we discuss some details for \Belyi\ maps of genus $1$.

\subsection{Newton's method}

Let $(X,\phi)$ be a \Belyi\ map with $X$ of genus $1$ defined by $X\colon y^2 =
f(x) = x^3 - 27c_4 x - 54c_6$. In our numerical method (see section
\ref{sec:overview}, or the Genus 1 subsection of \cite[\S5]{KMSV}), we compute
a numerical Weierstrass $X$ and \Belyi\ map $\phi$ on $X$ to arbitrary
precision.

Klug--Musty--Schiavone--Voight \cite[Example 5.28]{KMSV} describe how to use
Newton's method in the case of genus $0$ to achieve very accurate
approximations of the coefficients of the \Belyi\ map, allowing us to quickly
pass from tens of digits of precision to tens of thousands.  We now explain how
Newton's method can be extended to the case of genus $1$ \Belyi\ maps, ironing
out some wrinkles.



Let $P = (x_P, y_P) \in X(\C)$ be an affine point and let $t \colonequals x -
x_P$ and $s \colonequals y - y_P$. Insisting that $\phi$ have a zero or pole of
a given order at $P$ imposes equations that can be determined by working in the
completed local ring $\widehat{\mathbb{C}[X]}_P$ as follows.

If $P$ is not a $2$-torsion point of $X$, then $t$ is a uniformizer for
$\widehat{\mathbb{C}[X]}_P$.  We solve for $s$ in terms of $t$ by substituting
$x = t + x_P$ and $y = s + y_P$ into the equation for $X$, thereby obtaining a
quadratic equation in $s$ whose solution is
\begin{equation} \label{eqn:non-2-torsion}
  s \colonequals -y_P
  + y_P \sqrt{1+\frac{t^3+3 x_P t^2 + (3 x_P^2 - 27 c_4) t}{y_P^2}} \, .
\end{equation}
If instead $P$ is a $2$-torsion point, then $s$ is a uniformizer for
$\widehat{\mathbb{C}[X]}_P$; substituting as above, we obtain a cubic equation
in $s$, which we solve via Hensel lifting.  In either case, we may express the
numerator and denominator of $\phi$ as power series in the local parameter.
Once this has been accomplished, we obtain the equations imposed by a zero
(resp., pole) at $P$ of order $e_P$ by insisting that the first $e_P$
coefficients of the series for the numerator (resp., denominator) of $\phi$
vanish.

Newton's method has proven invaluable in our computations: it has allowed us to
compute genus $1$ maps that were previously out of reach, and has also sped up
our computations considerably. 

\subsection{Example}

We illustrate the above method with an example.

\begin{exm}
  Consider the passport $(1,S_7,(6^1 1^1, 6^1 1^1, 2^2 3^1))$ of size $13$.
  Its pointed refinement taking the $6$-cycle over $0$ also has size $13$.  A
  representative permutation triple is
  \begin{equation}
    \sigma_0 = (1\ 2\ 3\ 4\ 5\ 6), \qquad
    \sigma_1 = (2\ 7\ 6\ 3\ 4\ 5), \qquad
    \sigma_\infty = (1\ 7\ 2)(3\ 5)(4\ 6).
  \end{equation}
  This ramification data and a Riemann--Roch calculation implies that $\phi$
  can be written as $\phi = \phi_0/\phi_\infty$ for $\phi_0 \in \L(2 \infty)$
  and $\phi_\infty \in \L(8 \infty)$.  (For details, see section
  \ref{sec:computingbelyihyper} below.)
  Since $1, x$ and $1, x, y, x^2, xy, \ldots, x^4$ are bases for $\L(2 \infty)$
  and $\L(8 \infty)$, respectively, pulling out leading coefficients and
  changing notation, we write
  \begin{equation} \label{eqn:undetermined}
    \phi = u \frac{\phi_0}{\phi_\infty}
         = u \frac{a_0 + x}{b_0 + b_2 x + b_3 y + \cdots + b_7 x^2 y + x^4}
  \end{equation}
  for some $a_0, a_2, b_0, b_2 \ldots, b_7 \in \Qbar \subset \C$.
  Computing with $40$ digits of precision (displaying $5$), we find after 20
  seconds on a standard CPU the initial approximation for $X$ and $\phi$.
  After normalizing as in section \ref{sec:ahhhTheta} to
  obtain $b_7(=b_8)=1$, we obtain
  \begin{equation}
    \begin{aligned}
      c_4,c_6 &\approx -0.00031, 0.0000035 \\
      \phi &\approx 0.0024\frac{-0.18587 + x}{-0.00042 + 0.00112 x + \dots + 0.03839 x^3 + x^2y + x^4}.
    \end{aligned}
  \end{equation}
  Let $P = (x_P, y_P)$ be the point corresponding to the $3$-cycle in
  $\sigma_\infty$.  Since $P \in X(\C)$, our first equation is $y_P^2 = x_P^3 -
  27 c_4 x_P - 54 c_6$.  Computing $s$ as in \eqref{eqn:non-2-torsion}, we find
  \begin{equation}
    \begin{aligned}
      s &= \frac{\frac{3}{2} x_P^2 - \frac{27}{2} c_4}{y_P} t
        + \frac{-\frac{9}{8} x_P^4 + \frac{81}{4} c_4 x_P^2 + \frac{3}{2} x_P y_P^2
        - \frac{729}{8} c_4^2}{y_P^3} t^2\\
        &\qquad + \frac{\frac{27}{16} x_P^6 - \frac{729}{16} c_4 x_P^4
        +\dots+ \frac{81}{4} c_4 x_P y_P^2 + \frac{1}{2} y_P^4
        - \frac{19683}{16} c_4^3}{y_P^5} t^3 + O(t^4) \, .
    \end{aligned}
  \end{equation}
  Substituting $x = t + x_P$ and $y = s + y_P$ into the above expression for
  $\phi_\infty$ yields
  \begin{equation}
    \begin{aligned}
      \phi_\infty &= x_P^4 + x_P^3 b_6 + x_P^2 y_P b_7 + x_P^2 b_4 + x_P y_P b_5
                   + x_P b_2 + y_P b_3 + b_0\\
      & \qquad + \left(\textstyle{\frac{3}{2} x_P^4 b_7 + 4 x_P^3 y_P + \frac{3}{2} x_P^3
        +\dots+b_5 + y_P b_2 - \frac{27}{2} c_4 b_3}\right)\frac{t}{y_P}\\
      & \qquad + \left(\textstyle{-\frac{9}{8} x_P^6 b_7 - \frac{9}{8} x_P^5 b_5
        +\dots+\frac{729}{8} c_4^2 b_3}\right)\frac{t^2}{y_P^3} + O(t^3) \, .
    \end{aligned}
  \end{equation}
  To impose the condition that $\phi$ has a pole of order $3$ at $P$, we set
  the first three coefficients of $\phi_\infty$ equal to zero, giving $3$
  relations.

  Proceeding similarly with the other ramification points, we obtain $22$
  polynomial equations in the $23$ variables $u, c_4, c_6, a_0, b_0, \ldots,
  b_7$ and $x_P, y_P$ for each of the ramification points, other than the point
  corresponding to the cycle containing $1$ in $\sigma_0$. (The point
  corresponding to this cycle is $\infty$, and we have already imposed the
  condition that $\phi$ vanishes to order $6$ at $\infty$ by taking $\phi_0 \in
  \L(2\infty)$ and $\phi_\infty \in \L(8 \infty)$.) This system is
  underdetermined, so in order to apply Newton's method, we must find at least
  one more equation.  We observe that although $\phi$ is a degree $7$ map,
  $\phi_\infty$ has degree $8$, so there must be a common zero of $\phi_0$ and
  $\phi_\infty$. Calling this point $P_s = (x_s, y_s)$, we obtain $3$ more
  equations
  \begin{equation} \label{eqn:extraeqns}
    \begin{gathered}
      y_s^2 = x_s^3 - (27 c_4 x_s - 54 c_6) \qquad 0 = \phi_0(P_s) = a_0 + x_s \\
      0 = \phi_\infty(P_s) = b_0 + b_2 x_s + b_3 y_s + \cdots + b_7 x_s^2 y_s + x_s^4 \, .
    \end{gathered}
  \end{equation}
  We have adjoined two more variables $x_s, y_s$ and produced three more
  equations to ensure non-degeneracy. This produces a system of $25$ equations
  in $25$ variables. Applying Newton's method to this system, in 16.20 seconds
  we obtain approximations of coefficients with $2000$ digits of precision,
  which allows us to recognize the coefficients of $\phi$ as algebraic numbers.
  After a change of variables to reduce the size of the output, we find the
  elliptic curve
  \begin{equation}
    X \colon y^2 = x^3 - (24\nu+75)x + \textstyle{\frac{1}{2}}(-657\nu^2-1014\nu+3278)
  \end{equation}
  and \Belyi\ map $\phi=u\phi_0/\phi_\infty$ where $u=1/(2^9 3^2)$ and
  \begin{equation*}
      \begin{aligned}
        \phi_0 &= (-419\nu^2 - 358 \nu + 2947)+49x \\
        \phi_\infty &= (-806361\nu^2 - 724014\nu + 5449304) +(-3150\nu^2 - 15652\nu + 84560)x \\
        &\qquad     +(- 11310\nu^2 + 17940\nu + 118656)y     +(-33180\nu^2 + 74760\nu - 55104)x^2 \\
        &\qquad     +(59556\nu^2 - 189336\nu + 233856)xy     +(5166\nu^2 - 16380\nu + 20720)x^3 \\
        &\qquad     +(-59022\nu^2 + 184980\nu - 225792)x^2y  +(25557\nu^2 - 80122\nu + 97832)x^4
      \end{aligned}
  \end{equation*}
  over the number field $\Q(\nu)$ where $\nu^3 - 6 \nu + 12 = 0$.  It turns out
  that this passport decomposes into two Galois orbits, one of size $3$ as
  shown above, and the other of size $10$.  The coefficients of the \Belyi\ map
  for the size $10$ orbit are too large for us to display here, but it is
  defined over the number field $\Q(\mu)$ where
  \begin{equation}
    \mu^{10} - 2  \mu^9 + 15  \mu^8 - 78  \mu^7 + 90  \mu^6 + 48  \mu^5 + 90
    \mu^4 - 78  \mu^3 + 15  \mu^2 - 2  \mu + 1 = 0 \, .
  \end{equation}
\end{exm}

\begin{rmk}
  The ``extra zero" phenomenon in \eqref{eqn:extraeqns} is typical; it can be
  avoided in the special case when $0$ is totally ramified
  (i.e., when $\sigma_0$ is a $d$-cycle).
\end{rmk}


\section{Hyperelliptic curves} \label{sec:highergenus}

We now discuss some issues and improvements for hyperelliptic curves.

\subsection{Setup}

Recall that a hyperelliptic curve of genus $g \geq 2$ over $K$ has a model
\begin{equation} \label{hyperelliptic_eqn}
  X\colon y^2 + u(x)y = v(x)
\end{equation}
where $\deg(u) \leq g+1$ and $\deg(v) \leq 2g+2$. Letting $f(x) \colonequals
u(x)^2+4v(x)$, we have $f(x)$ separable with $\deg f(x)=2g+1$ or $2g+2$; we
refer to the model as \defi{even} or \defi{odd} according to the parity of
$\deg f(x)$. Note that an odd model has the single point $\infty = (1:1:0)$ at
infinity while an even model has two, $\infty' = (1:\sqrt{f_0}:0)$ and $\infty
= (1:-\sqrt{f_0}:0)$ where $f_0$ is the leading coefficient of $f(x)$
(I.e., the point $\infty$ is a Weierstrass
point if and only if the model is odd.)  In constructing the \Belyi\ map, in
both cases we take $\infty$ to be the marked point (around which we expand
series), and by convention it corresponds to the cycle containing $1$ in
$\sigma_0$.

\subsection{Numerical test for hyperellipticity} \label{sec:numhyperl}

Let $\Gamma$ be a triangle subgroup with $X=X(\Gamma)$ of genus $g \geq 2$.  We
test if $X$ is numerically hyperelliptic (in the sense the curve appears to be
hyperelliptic to the precision computed) as follows.  First, we
compute power series expansions of an \emph{echelonized} basis $f_1, f_2, \ldots, f_g$
of $S_2(X(\Gamma))$.  We have an isomorphism $S_2(X(\Gamma)) \cong \Omega(X(\Gamma))$ given by
$f(z) \mapsto f(z) \, \mathrm{d}z$
where $\Omega(X(\Gamma))$ is the $\C$-vector space of holomorphic differential
$1$-forms on $X(\Gamma)$.  If $X$ is hyperelliptic with model as in
\eqref{hyperelliptic_eqn}, since $f_1, \ldots, f_g$ is an echelonized basis we
have the further isomorphism
\begin{equation}
  \begin{aligned}
    \Omega(X(\Gamma)) &\overset{\sim}{\to} \Omega(X)\\
    f_i(z) \, \mathrm{d}z &\mapsto x^{g-i} \frac{\mathrm{d}x}{y}
  \end{aligned}
\end{equation}
for $i = 1, \ldots, g$.  Thus, to recover $x,y$ defined on
$X(\Gamma)$, we can take
\begin{equation} \label{eqn:xandy}
  x \colonequals f_1/f_2 \qquad \qquad y \colonequals x'/f_g
\end{equation}
where $x'$ denotes the derivative of $x$ with respect to $w$ (the coordinate in
the hyperbolic disc).  If the model is odd, then $\ord_\infty x=-2$ and $\ord_\infty y=-(2g+1)$;
if the model is even, then $\ord_\infty x = -1$ and $\ord_\infty y = -(g+1)$.

Consider the rational map $X(\Gamma) \to \mathbb{A}_{\C}^2$ with coordinates
$x,y$.  Using numerical linear algebra, we test if there is an approximate
linear relation among
\begin{equation}
  1, x, \ldots, x^{2g+2}, y, xy, \ldots, x^{g+1}y, y^2 \in \C[[w]].
\end{equation}
If there is such a relation, we obtain a rational map from $X$ to a
hyperelliptic curve $X' \subseteq \mathbb{A}^2$.  If $g(X') = g(X)$, then the
Riemann--Hurwitz formula implies that this map is birational, hence $X'$ is a
model of $X$ as in \eqref{hyperelliptic_eqn}.  If no such relation exists, then
we conclude that $X$ is not numerically hyperelliptic.


\subsection{Computing a hyperelliptic Belyi map} \label{sec:computingbelyihyper}

Suppose now that $X$ is hyperelliptic with model as in
\eqref{hyperelliptic_eqn}.  We compute the expression of the \Belyi\ map $\phi$
as a rational function in $x$ and $y$ roughly as follows.
\begin{enumerate*}
  \item Determine an appropriate Riemann--Roch space $\L(D)$.

  \item Compute a basis of $\L(D)$ in terms of $x$ and $y$.

  \item Using numerical linear algebra, express $\phi$ as a linear combination
    of functions in this basis.
\end{enumerate*}


We make this precise as follows, following Javanpeykar--Voight \cite[Lemma 3.2]{JV}.
Let $\sigma = (\sigma_0, \sigma_1, \sigma_\infty)$ be a transitive permutation
triple of degree $d$ with corresponding hyperelliptic 
\Belyi\ map $(X, \phi)$, and let $g$ be
the genus of $X$. Let $s$ be the length of the cycle containing $1$ in
$\sigma_0$ and let $k_1, \ldots k_r$ be the lengths of the remaining cycles in
$\sigma_0$.  Then the divisor of poles of $1/\phi$ is $\div_\infty(1/\phi) =
s\infty + \textstyle{\sum_{i=1}^r} k_i P_i$ for some points $P_1, \ldots, P_r \in X(\C)$.
Since we do not have control over the points $P_1, \ldots,
P_r$, we ``cancel" these poles by multiplying $\phi$ by a suitable function
$\phi_0$ that has zeroes at $P_1, \ldots, P_r$ and has poles only at $\infty$.
Such a $\phi_0$ will belong to the space $\L(D) \subseteq \L(t\infty)$ where
\begin{equation}
D \colonequals -\textstyle{\sum_{i=1}^r} k_i P_i + t \infty
\end{equation} for some (as of yet undetermined) $t \in
\mathbb{Z}_{\geq 0}$. Once we have obtained $\phi_0$, then $\phi_0/\phi \in
\L((s+t)\infty)$.  As we will describe in the next step, we can write down a basis
for Riemann--Roch spaces for divisors of the form $m \infty$.  This allows us to
compute $\phi_0$ and $\phi_\infty := \phi_0/\phi \in \L((s+t)\infty)$ with
respect to this basis. Thus we have $\phi = \phi_0/\phi_\infty$ for some
$\phi_0 \in \L(t\infty)$ and $\phi_\infty \in \L((s+t)\infty)$.

%
%
%

It remains to determine a value of $t$ so that such a $\phi_0$ exists.  To do
this, we apply Riemann--Roch to the divisor $D$.  Since $\textstyle{\sum_{i=1}^r} k_i = d-s$, this yields
\begin{equation} \label{eqn:RR}
  \begin{aligned}
    \ell(D) - \ell(K_X-D) = 1 - g + \deg(D) = 1 - g + (s - d + t)
  \end{aligned}
\end{equation}
where $K_X$ is a canonical divisor of $X$.  To ensure the existence of a
nonzero $\phi_0$ as above, we must have $\ell(D) \geq 1$.  By (\ref{eqn:RR}),
this holds if $1 - g + s - d + t \geq 1$, i.e., if $t \geq d - s + g$.  Thus we
may take $t = d - s + g$.  (This conclusion actually does not require $X$ to be
hyperelliptic.) 

%

Next, we explain how to compute bases for $\L(t\infty)$ and $\L((s+t)\infty)$
as in step 2.  In the case of an odd model, this basis is particularly simple:
$x$ and $y$ have poles at $\infty$ of orders $2$ and $2g+1$, respectively, so
\begin{equation}
  1, x, x^2, \ldots, x^{\lfloor m/2 \rfloor}, y, xy, \ldots, x^{\lfloor \frac{m - (2g+1)}{2} \rfloor} y
\end{equation}
is a basis for $\L(m \infty)$.  In the case of an even model the situation is
more complicated.  Now $x,y \not\in \L(m \infty)$ because they have poles at
$\infty'$.  We compute a basis for $\L(m \infty)$ as follows.  Since $x$ has a
simple pole at $\infty'$ then $t = 1/x$ has a simple zero, hence is a
uniformizing parameter at $\infty'$.
Working in the completed local ring $\widehat{\mathcal{O}}_{X,\infty'} \simeq
\C[[t]]$, we can express $y$ as a Laurent series in $t$ via
\begin{equation} \label{eqn:yexpsqrt}
  y = \frac{1}{2}\left(-u(1/t) \pm \sqrt{u(1/t)^2 + 4 v(1/t)}\right).
\end{equation}
We use the series expansions $x(w),y(w)$ at $\infty$ to match the correct sign
in \eqref{eqn:yexpsqrt}.
For each $j \in \{0, \ldots, m - (g+1)\}$ we compute the Laurent tail $P_j \in
\C[1/t]=\C[x]$ of $x^{j}y$, so that $x^{j}y - P_j$ is holomorphic at $\infty'$.
In this way we obtain the basis
\begin{equation}
  1, y - P_{0}, xy - P_1, \ldots, x^{m-(g+1)}y - P_{m-(g+1)}
\end{equation}
for $\L(m \infty)$.

\begin{exm}
  We illustrate the above procedure with an example.  Consider the passport
  $(2,G,(6^1,6^1,3^2))$, where $G \colonequals 2A_4(6) \simeq A_4 \times
  C_2$. The passport (and pointed passport) are size $1$, with representative
  triple
  \begin{equation}
    \sigma_0 = (1\ 6\ 2\ 4\ 3\ 5), \quad
    \sigma_1 = (1\ 3\ 5\ 4\ 6\ 2), \quad
    \sigma_\infty = (1\ 3\ 5)(2\ 4\ 6).
  \end{equation}
  Computing the coordinate functions $x,y$ as in \eqref{eqn:xandy} to $50$
  digits (displaying $5$), we find approximate series
  \begin{equation}
    \begin{aligned}
      x &\approx 0.99999 w^{-1} - 0.79370 w - 0.31498 w^3 + O(w^4)\\
      y &\approx -0.99999 w^{-3} - 0.79370 w^{-1} - 0.94494 w - 0.02142 w^3 + O(w^4) \, .
    \end{aligned}
  \end{equation}
  Since the series for $y$ has a pole of
  order $3 = g+1$, we are in the case of an even model.  Forming the matrix of
  coefficients of the monomials
  \begin{equation}
    1, x, x^2, x^3, x^4, x^5, x^6, y, xy, x^3 y, y^2 \, ,
  \end{equation}
  we find a hyperelliptic equation as in \eqref{hyperelliptic_eqn} with $u = 0$
  and
  \begin{equation}
    v \approx 1.00000  x^6 + 6.34960 x^4 + 15.11905 x^2 + 11.99999
  \end{equation}
  This gives the local expansion
  \begin{equation}
    \begin{aligned}
      y &= \sqrt{v(1/t)} = \sqrt{1.00000  t^{-6} + 6.34960 t^{-4} + 15.11905 t^{-2} + 11.99999}\\
        &= 1.00000 t^{-3} + 3.17480 t^{-1} + 2.51984 t - 1.99999 t^3 + O(t^4) \, .
    \end{aligned}
  \end{equation}
  Thus the Laurent tail of $y$ is $1.00000 x^{3} + 3.17480 x$, and the first
  nonconstant element of our basis for $\L(m \infty)$ for $m \geq 3$ is thus
  \begin{equation}
    \begin{aligned}
      y - (&1.00000 x^{3} + 3.17480 x)\\
        &\approx -2.00000 w^{-3} - 1.58740 w^{-1} + 0.62996 w - 0.04285 w^3 + O(w^4)
    \end{aligned}
  \end{equation}
  and we can compute the remaining elements of the basis similarly. Proceeding
  as explained above, we obtain the \Belyi\ map
  \begin{equation}
    \phi(x,y) = \frac{x^4 + 2x^2 + xy + 1}{2(x^2 + 1)^2}
  \end{equation}
  defined on the hyperelliptic curve $X\colon y^2 = x^6 + 4x^4 + 6x^2 + 3$.
\end{exm}


%
%
\section{Database} \label{sec:database}

\subsection{Technical description}
Our database is organized by passports as computed in Algorithm
\ref{alg:passportalg}. For each passport we store basic information such as
degree, genus, ramification indices, and the monodromy group.
We also store the
automorphism group and passport representatives, as well as their
pointed counterparts.
After computing equations for every \Belyi\
map in a passport, we store the \Belyi\ maps, curves, the fields
over which they are defined,
and the associated complex embedding.
We then
partition the pointed passport representatives into Galois orbits obtained from
this information. Lastly, the numerical power series and information to recover
the normalization in Section \ref{sec:overview} Step 4 are also saved.

\subsection{Running time}

Since our numerical method for computing equations sometimes requires a workaround for corner cases,
we do not have detailed information about the total running time.  
To give a rough idea of the running time, we consider some examples.
In \eqref{table:runtime} we list the approximate CPU time to
compute \emph{one} \Belyi\ map in the listed passport, with power series
computed to the specified number of decimal digits of precision and then precision obtained in Newton iteration.
\begin{equation} \label{table:runtime}
  \begin{tabular}{l|l|l|l}
    Passport&Size&Precision (Newton)&CPU Time\\
    \hline\hline
    $(0,A_9,(5^12^2,3^3,4^12^11^3))$&2&20 (1000) &7s\\
    \hline
    $(0,S_9,(7^12^1,4^12^11^3,4^12^21^1))$&23&20 (16000) &2m46s \\
    \hline
    $(1,A_7,(7^1,3^12^2,3^12^2))$&2&30 (1000)&23s\\
    \hline
    $(1,S_7,(5^12^1,5^12^1,3^12^2))$&4&40 (1500)&2m48s\\
    \hline
    $(1,A_7,(7^1,4^12^11^1,4^12^11^1))$&22&20 (1500)&10s\\
    \hline
    $(2,\GL_3(\F_2),(7^1,7^1,3^21^1))$&4&20&4m59s\\
  \end{tabular}
\end{equation}
The current database of \Belyi\ maps
consists of approximately $240$MB of text files.

\subsection{Observations} \label{sec:interesting}

Having completed a large scale computation of \Belyi\ maps, it remains to
analyze our data.  

\begin{itemize}
  \item
    The largest passport sizes in each degree are:
    \begin{equation} \label{table:passportsizes}
      \begin{tabular}{l|c|c|c|c|c|c|c|c}
        Degree&$\leq 4$&$5$&$6$&$7$&$8$&$9$&$10$&$11$\\
        \hline
        Passport size&$1$&$3$&$8$&$38$&$177$&$1260$&$8820$&$72572$
      \end{tabular}
    \end{equation}
  \item The largest degree number field arising as a field of definition of a \Belyi\ map in our database
    occurs for the passport $(1, S_7, (6^1 1^1, 6^1 1^1,
    4^1 2^1 1^1))$ which is irreducible of size $32$.  
    This degree $32$ number field has discriminant $2^{68} 3^{27} 5^9 7^{15}$
    and Galois group $\Z/2\Z\wr S_{16}$.


  \item
  	  The passport $(2, A_7, (7^1, 7^1, 5^1 1^1 1^1))$ provides an example of a highly reducible passport: it has size $24$ and decomposes into $6$ Galois orbits of sizes $1, 2, 3, 4, 6,$ and $8$.  The associated number fields are $\Q$, and those with defining polynomials $x^2 - x - 5, \quad x^3 + 2 x - 2, \quad x^4 - 2 x^3 - 2 x^2 + 3 x - 3, \quad
x^6 - 2 x^4 - 5 x^3 - 2 x^2 + 1$, and $x^8 - 4 x^7 + 14 x^5 - 35 x^4 + 42 x^3 - 126 x^2 + 108 x + 135$.  
  \item
    There are $262$ passports with degree $d\leq 7$.
    We have computed equations for all \Belyi\ maps in $255$
    of these passports and found that $37$ are reducible.
    For a passport $\mathcal{P}$ of size $l$,
    the Galois action determines a partition of $l$
    with parts $l_1,\dots,l_r$.
    To measure the irreducibility of $\mathcal{P}$,
    define
    \begin{equation}\label{eqn:wt}
      \mathrm{w}(\mathcal{P})\colonequals
      \begin{cases}
        1, &\text{ if } l=1;\\
        (l-1)^{-2}\textstyle\sum_{i=1}^r (l_i-1)^2, &\text{ if }l\geq 2.
      \end{cases}
    \end{equation}
    Let $\mathscr{P}_d$ be the set of passports
    with degree no larger than $d$ and define
    \begin{equation}\label{eqn:belyiconstant}
      \beta(d)\colonequals
      (\#\mathscr{P}_d)^{-1}\textstyle{\sum_{\mathcal{P}\in\mathscr{P}_d}\mathrm{w}(\mathcal{P})}.
    \end{equation}
    From the database we find that
    $\beta(d) = 1$ for $d\le 4$,
    $\beta(5) \approx 0.9393$,
    $\beta(6) \approx 0.9444$,
    and
    $0.8779<\beta(7)<0.9046$.
\end{itemize}





\end{document}